\documentclass[twoside]{article}

\usepackage{amsmath}
\usepackage{amssymb}
\usepackage{amsthm}
\usepackage{amsfonts}
\usepackage{fancyhdr}
\usepackage{hyperref}
\usepackage{color}
\usepackage{enumitem}
\usepackage{todonotes}
\usepackage{mathabx} 
\usepackage{changes}
\usepackage{empheq}
\usepackage{mathrsfs}

\usepackage{mathtools}
 \mathtoolsset{showonlyrefs}

\theoremstyle{definition}
\newtheorem{Def}{Definition}[section]

\theoremstyle{plain}

\newtheorem{lem}[Def]{Lemma}
\newtheorem{teo}[Def]{Theorem}
\newtheorem{Propos}[Def]{Proposition}

\theoremstyle{remark}
\newtheorem{remark}{Remark}[section]
 
\newcommand{\orlnor}{\|_{L^{\phi}}}

\newcommand{\lphi}{L^{\phi}}

\newcommand{\wphi}{W^{1,\phi}}
\newcommand{\wphit}{W^{1,\phi}_T}

\newcommand{\rr}{\mathbb{R}}

\makeatletter
\newcommand{\labitem}[2]{%
\def\@itemlabel{\textnormal{\textbf{(#1)}}}
\item
\def\@currentlabel{\textnormal{\textbf{(#1)}}}\label{#2}}
\makeatother
\makeatletter
\def\namedlabel#1#2{\begingroup
    #2%
    \def\@currentlabel{#2}%
    \phantomsection\label{#1}\endgroup
}
\makeatother

\title{Existence of periodic solutions for Hamiltonian inclusion systems  using Clarke duality \thanks{   The research has been supported by SECyT-UNRC, grant C561,  FCEyN-UNLPam, grant PI 91. M. and ANPCyT PICT 2019- 3837
}}
\author{
Stefania M. Demaria and  Fernando D. Mazzone \\
Dpto. de Matem\'atica, Facultad de Ciencias Exactas, F\'{\i}sico-Qu\'{\i}micas y Naturales\\
Universidad Nacional de R\'{i}o Cuarto\\
(5800) R\'{\i}o Cuarto, C\'ordoba, Argentina,\\
\url{sdemaria@exa.unrc.edu.ar},
\url{fmazzone@exa.unrc.edu.ar} \\
}

\date{}

\begin{document}
\maketitle
\begin{abstract}We prove the existence of periodic solutions for Hamiltonian differential inclusions under growth conditions involving a $G$-function.
\end{abstract}

 \noindent\textit{keywords:} Hamiltonian, inclusions, periodic, $G$-function.

\section{Introduction and main results}

In this paper, we establish the existence of solutions to the following boundary value problem for a Hamiltonian  differential inclusion:

\begin{equation}\label{problemadualfinal}
\left\{ \begin{array}{lcc}
             u'(t)\in J\partial  \mathcal{H}(t,u(t)) \quad \text{for a.e. } t \\
             \\ u(0)=u(T), \\
             \end{array}
   \right.
\end{equation}
where $ \mathcal{H}:[0,T]\times \mathbb{R}^{2n}\to \mathbb{R}$ is a convex function with respect to the second variable, $\partial  \mathcal{H}(t,u)$ denotes the subdifferential in the sense of convex analysis (see \cite{clarke2013functional}), and $J$ denotes the canonical symplectic matrix.

\begin{equation*}
J=\begin{pmatrix}
0_{n\times n} & I_{n\times n} \\
-I_{n\times n} & 0_{n\times n}
\end{pmatrix}.
\end{equation*}

It is well known that many problems in Lagrangian mechanics can be reformulated as Hamiltonian problems.

Following Trudinger (see \cite{trudinger1974imbedding}), we say that $\phi: \mathbb{R}^d \to [0,+\infty]$ is a \emph{$G$-function} if it satisfies the following properties: $\phi$ is convex, $\lim_{|x|\to\infty}\phi(x)=\infty$, $\phi(0)=0$, $\phi(-x)=\phi(x)$, $\phi$ is bounded in some neighborhood of $0$, and $\phi$ is lower semicontinuous.

Associated with $\phi$, we have the \emph{complementary function} $\phi^{\star}$, which is defined for $\xi\in \mathbb{R}^d$ as
\begin{equation}\label{eq:conjugada}
\phi^{\star}(\xi)=\sup\limits_{x\in \mathbb{R}^d } \langle \xi, x\rangle-\phi(x),
\end{equation}
where $ \langle \cdot, \cdot\rangle : \mathbb{R}^d \times \mathbb{R}^d \to\mathbb{R}$ is the usual dot product.

 Some elementary and useful properties satisfied by $G$-functions are:

\begin{enumerate}
\item[\namedlabel{G:supra_adit}{(P1)}] if $0\leq \lambda_1,\lambda_2$, then $\phi\left((\lambda_1+\lambda_2) x\right)\geq\phi(\lambda_1 x)+\phi(\lambda_2 x);$
\item[\namedlabel{G:young-ine}{(P2)}] $x\cdot y\leq \phi(x)+\phi^{\star}(y)$ (Fenchel's inequality);
\item[\namedlabel{G:young-eq}{(P3)}] $y\in\partial\phi(x)\Longleftrightarrow x\cdot y= \phi(x)+\phi^{\star}(y)\Longleftrightarrow x\in\partial\phi^\star(y)$ (Fenchel's identity);
\item[\namedlabel{G:d(2x)}{(P4)}] if $\phi$ is Fréchet differentiable, then $\langle \nabla \phi( x),x\rangle \leq \phi(2x)$, for every $x\in\rr^d$;
\item[\namedlabel{G:|x|<=G}{(P5)}] there exists $C>0$ such that $C|x|\leq \phi(x)+1$, for every $x\in \mathbb{R}^d$.
\end{enumerate}

We say that a $G$-function $\phi:\mathbb{R}^d\rightarrow [0,+\infty)$ satisfies the \emph{$\Delta_2$-condition}, and we denote $\phi \in \Delta_2$, if there exists a constant $C>0$ such that
\begin{equation}\label{delta2defi}\phi(2x)\leq C \phi(x)+1,\quad x\in \mathbb{R}^d.
\end{equation}
Note that this definition is equivalent to the classical one, i.e., there exist $r_0, C > 0$ such that $\phi(2x) \leq C\phi(x)$ for $\|x\| > r_0$. When the inequality $\phi(2x)\leq C \phi(x)$ is satisfied for every $x$, we say that $\phi$ satisfies the \emph{$\Delta_2$-condition globally}. We denote this by $\phi\in \Delta_2^G$.

The \emph{Orlicz space} $\lphi=L^{\phi}\left([0,T],\mathbb{R}^d\right)$ is defined by
\begin{equation}\label{espacioOrlicz}
\lphi:=\big\{ u\big| u \text{ is Bochner measurable and } \exists \lambda>0:  \int_0^T\phi(\lambda u) < \infty   \big\}.
\end{equation}

The Orlicz space $\lphi$, equipped with the \emph{Luxemburg norm}
\[
\|  u  \orlnor:=\inf \left\{ \lambda\bigg| \int_0^T\phi\left(\frac{u}{\lambda}\right) \,dt\leq 1\right\},
\]
is a Banach space.

It follows easily from the definition and elementary properties of $G$-functions that
\begin{equation}\label{eq:desi_norm_mod}
 \|  u  \orlnor\leq  \int_0^T\phi\left(u\right)dt+1.
\end{equation}

We define the \emph{Orlicz-Sobolev space} $\wphi=\wphi\left([0,T],\mathbb{R}^d\right)$ by
\[\wphi:=\big\{u| u \text{ is absolutely continuous, }  u'\in \lphi\big\}.\]

The space $\wphi$ is a Banach space when it is equipped with the norm
\begin{equation}\label{def-norma-orlicz-sob}
\|  u  \|_{\wphi}= \|  u  \|_{\lphi} + \|u'\orlnor.
\end{equation}
The subspace $\wphit$ of $\wphi$ is defined by
\[\wphit:=\wphi\cap \big\{u| u(0)=u(T)\big\}.\]
Note that $\wphit$ is a closed subspace of $\wphi$.

Following \cite{acinas2019clarke}, we introduce the following notion.
\begin{Def}\label{defsimplectica}
We say that a $G$-function $G:\mathbb{R}^{2n}\to [0,\infty)$ is \emph{symplectic} if $G^\star(Ju)=G(u)$ for all $u \in \mathbb{R}^{2n}$.
We say that it is \emph{semi-symplectic} if $G^\star(J\cdot)\prec G (\cdot)$, where $G_1 \prec G_2$ means that $\exists\hspace{0.1cm} K>0 \text{ and } C\geq 0$ such that, for all $u$, we have that $G_1(u)\leq G_2(Ku)+C.$
\end{Def}

In the literature (see \cite{meyer2017introduction}), the notions of \emph{symplectic matrix} and \emph{symplectic or canonical transformation} are more widely known. In Section \ref{sec:symplec}, we will recall these concepts and discuss the relationship between them and the notion of symplectic $G$-function introduced above.
In order to establish our main theorem, we require the following lemma, which will be proved in Section \ref{sec:lemas}.
\begin{lem}\label{lem:form_quad}
Let $G:\mathbb{R}^{2n}\to [0,\infty)$ be a semi-symplectic $G$-function. Then there exist positive constants $C_1, C_2$, depending only on $G$ and $T$, such that for every $u \in W^{1,G}_T([0,T], \mathbb{R}^{2n})$ we have

\begin{equation}\label{inec16}
 \int_0^T \left<Ju',u\right>dt\geq -C_1\int_0^T G(Tu')dt-C_2.
\end{equation}

\end{lem}

As is customary, we will use the decomposition $u=\overline{u}+\widetilde{u}$ for a function $u\in L^1([0,T],\rr^d)$, where $\overline{u} =\frac1T\int_0^T u(t)\, dt$ and $\widetilde{u}=u-\overline{u}$. The following proposition will be useful in the sequel; see \cite[Prop. 3.4]{demaria2025} for a proof.
\begin{Propos}\label{propos:noprmas_equi} The norm $\|u\|_{\wphi}':=|\overline{u}|+\|u'\|_{\lphi}$ is equivalent to $\|\cdot\|_{\wphi}$ on the space $\wphit$.
\end{Propos}

We present our main theorem.

\begin{teo}\label{teominimoacciondual}
Let $G:\mathbb{R}^{2n}\to[0,\infty)$ be a differentiable $G$-function, with $G^\star$ semi-symplectic and $G^\star \in \Delta_2$. Suppose that $ \mathcal{H}:[0,T]\times \mathbb{R}^{2n}\to \mathbb{R}$ is such that $ \mathcal{H}(t,\cdot)$ is convex for almost every $t \in [0,T]$, $ \mathcal{H}(\cdot,u)$ is measurable for each fixed $u$, and the following conditions hold:
\begin{enumerate}
    \item There exists $ \xi \in L^{G^\star}([0,T], \mathbb{R}^{2n})$ such that for all $y \in \mathbb{R}^{2n}$ and for almost every $t \in [0,T]$
    \begin{equation}\label{eq:cota_inf_H}
     \mathcal{H}(t,y)\geq \left<\xi(t),y\right>.
    \end{equation}

    \item There exists $\Lambda$, with $\Lambda^{-1}>T \max\{1,C_1/2\}$, where $C_1$ is the constant obtained from inequality (\ref{inec16}), and $\alpha \in L^1([0,T], \mathbb{R})$ such that, for all $y \in \mathbb{R}^{2n}$ and for almost every $t \in [0,T]$, one has

     \begin{equation}\label{eq:cota_sup_H}
       \mathcal{H}(t,y)\leq G(\Lambda y) + \alpha(t).
     \end{equation}
    \item  $\int_0^T  \mathcal{H}(t,y)\,dt\to \infty \quad \text{as } |y|\to \infty.$
\end{enumerate}

Then there exists $u\in W^{1,G}_T([0,T],\mathbb{R}^{2n})$ which is a solution of the Problem
\eqref{problemadualfinal} and it holds that $v=-J\tilde{u}$ minimizes the dual action
\begin{equation}\label{eq:accion_dual}
    \chi(v)=\int_0^T \frac{1}{2}\left<Jv',v\right>+ \mathcal{H}^\star(t,v')\,dt.
\end{equation}

\end{teo}

  It should be noted that Theorem \ref{teominimoacciondual} is a generalization of \cite[Th. 5.1]{acinas2019clarke}. In the latter paper, the existence of periodic solutions for Hamiltonian differential systems of equations was established, under the assumption that the Hamiltonian function is controlled by means of a symplectic $G$-function.
\section{Symplectic $G$-functions}\label{sec:symplec}

In this paper, we need $G$-functions that interact with the symplectic structure (see \cite[Section 3.2]{meyer2017introduction}) of $\mathbb{R}^{2n}$. We recall the meaning of these terms.

\begin{Def}\label{def:espaciosimplectico}
A pair $(X,\omega)$, where $X$ is a Banach space and $\omega$ is a bilinear form on $X\times X$, is called a \emph{symplectic space} if $\omega$ is \emph{alternating} ($\omega[u,v]=-\omega[v,u]$) and \emph{non-degenerate} (i.e., if $\omega[u,v]=0$ for all $v$, then $u=0$).
\end{Def}

It can be shown that if $X$ is finite-dimensional and symplectic, then the dimension of $X$ is even. The standard symplectic structure on $\mathbb{R}^{2n}$ is given by the bilinear form
$$\omega(x,y)=\langle x,Jy\rangle.$$

\begin{Def}\label{def:matrizsimplectica} (see \cite[Ch. 6]{meyer2017introduction})
A matrix $A$ is said to be \emph{symplectic} if and only if $A^t J A = J$. The set of all symplectic matrices is denoted by $S_p(2n,\mathbb{R})$. A transformation $T$ of an open set $\Omega\subset \mathbb{R}^n$ into $\mathbb{R}^n$ is called a \emph{symplectic transformation} if its Jacobian matrix $DT$ is symplectic.
\end{Def}
The set $S_p(2n,\mathbb{R})$ forms a subgroup of the \emph{general linear group} $GL(2n,\mathbb{R})$, which is called the \emph{symplectic group}.

Symplectic transformations, also called canonical transformations, are important in mechanics because they preserve the Hamiltonian structure of the equations \cite[Th. 9.2]{betounes2009differential}.

We shall show the relationship between Definitions \ref{defsimplectica} and \ref{def:matrizsimplectica} in the particular case where $G$ is the \emph{quadratic function}
\begin{equation}\label{eq:cuadratica}
  G(x)=\frac{1}{2}\langle Ax, x \rangle,\quad \text{with } A\in \mathbb{R}^{2n\times 2n}.
\end{equation}

In this definition, we may assume $A=A^t$, since
$$\left \langle\frac{1}{4}\left(A^t+A\right) x, x \right\rangle= \frac12\langle Ax, x\rangle.$$
We note that $D G (x)= A x$ and $D^2G=A$. According to \ref{G:young-eq}, we have $G ^\star (y)= \langle y,x\rangle-G(x)$ where $y=D G(x)=Ax$. Therefore
\[G^\star(y)= \frac{1}{2}\langle A^{-1}y, y\rangle.\]

To ensure that $G$ is a $G$-function, one must have $G(x)>0$ for $x\neq 0$. Thus $A$ must be \emph{positive definite}\index{matrix!positive definite}, i.e., $\langle Ax, x\rangle > 0$ for $x \neq 0$. Since $D^2 G =A$, by \cite[Theorem C, p. 100]{ArthurWayneRoberts}, the positivity of $A$ is necessary and sufficient for the strict convexity of $G$. Moreover, the positivity of $A$ implies that $\exists \hspace{0.1cm} k>0$ such that $G(x)=\langle Ax, x\rangle \geq k \|x\|^2$, so $\frac{G(x)}{\|x\|}\to \infty$ as $\|x\|\to \infty$, and since $G(x)\leq \|A\| \|x\|^2$, we have $\frac{G(x)}{\|x\|}\to 0$ as $\|x\|\to 0$. As is well known, a function that satisfies these conditions is called an \emph{$N$-function} (see \cite{KR}). Every $N$-function is a $G$-function.

We now highlight the following result.

\begin{Propos}
Let $G$ be defined by \eqref{eq:cuadratica} with $A=A^t$ positive definite. Then the following are equivalent:
 \begin{enumerate}
  \item\label{it:sym1} $G$ is a symplectic $N$-function,
  \item $A\in S_p(2n,\mathbb{R}).$
 \end{enumerate}
\end{Propos}

\begin{proof}
Setting $v=Ju$ and $B=A^{-1}- J^t AJ$, we deduce

\begin{eqnarray*}
     G^\star(v)-G(u)&=& \frac{1}{2}\langle A^{-1}v,v\rangle - \frac{1}{2}\langle Au, u\rangle \\ &=&
    \frac{1}{2}\langle A^{-1}v,v\rangle -\frac{1}{2}\langle J^t AJJu, Ju\rangle
    \\&=&
    \frac{1}{2}\langle A^{-1}v,v\rangle -\frac{1}{2}\langle J^t AJv, v\rangle
     \\ &=&    \frac{1}{2}\langle \left(A^{-1}- J^t AJ\right)v, v\rangle
      = \frac{1}{2}\langle Bv, v\rangle.
\end{eqnarray*}

If $A\in S_p(2n,\mathbb{R}) $, then $A^tJ A =J$ and $A=A^t$ imply $A^{-1}= J^t A J$, which yields $ G^\star(Ju)=G(u)$, showing that $G$ is symplectic. Conversely, if $G$ is symplectic, then $ G^\star(v)-G(u)=0$, which implies $\langle Bv, v\rangle=0$ for all $v\in\mathbb{R}^{2n}$. This shows that $B$ is antisymmetric. Since $A$ is symmetric, so is $A^{-1}$; hence $B$ is symmetric as well. Thus $B=0$, so $A^{-1}= J^t AJ$ and $A\in S_p(2n,\mathbb{R})$.
\end{proof}

One may ask whether there exist positive definite, symmetric, and symplectic matrices $A$. By \cite[Prop. 1.1.1]{habermann2006introduction}, $S_p(2n,\mathbb{R})$ contains the subgroup
\[D(n)=\left\{ \begin{pmatrix}
C & 0 \\
0 & \left(C^t\right)^{-1}
\end{pmatrix}: C \in GL(n, \mathbb{R})\right\}.\]
We can find matrices $A\in D(n)$ with the desired properties. Since $A$ must be positive definite and symmetric, a suitable choice is
$$A= \begin{pmatrix}
C & 0 \\
0 & C^{-1}
\end{pmatrix}, $$
where $C=C^t$ and $C$ is positive definite. Then, for $u=(x,y)$,
\[G(u)=\frac{1}{2}\langle Au,u\rangle = \frac{1}{2}\left[\langle x, C x \rangle + \langle y, C^{-1} y \rangle\right],\]
which provides an example of a symplectic $N$-function.

Defining $\phi:\mathbb{R}^n\to [0,\infty)$ by $\phi(x)=\frac{1}{2}\langle x, C x \rangle$, we have $\phi^\star(y)=\frac{1}{2}\langle y, C^{-1} y \rangle$, so the previous example can be written as
\[
 G(u)=\phi(x)+\phi^\star(y).
\]
It is easy to see that for any $G$-function $\phi:\mathbb{R}^n\to [0,\infty)$, this construction yields an example of a symplectic $G$-function.
\begin{Propos}\label{incrustsimplec}
Let $G$ be semi-symplectic. Then $J$ induces an embedding of $L^G([0,T], \mathbb{R}^{2n})$ into $L^{G^\star}([0,T], \mathbb{R}^{2n})$.
More precisely, for any $K>0$ and $C\geq 0$ such that $G^\star(Jx)\leq G(Kx)+C$ for all $x\in \mathbb{R}^{2n}$, we have that, for every $u\in L^G$,
 \[\|Ju\|_{L^{G^\star}}\leq K(CT+1)\|u\|_{L^G}.\]
\end{Propos}

\begin{proof}
If $\|u\|_{L^{G}}=1$, then $\int_0^T G(u)\,dt\leq 1$ and
 \begin{eqnarray*}
 \int_0^T G^\star\left(\frac{Ju}{K(CT+1)}\right)dt
 &\leq& \frac{1}{CT+1} \int_0^T G^\star\left(\frac{Ju}{K}\right)dt\\
 &\leq& \frac{1}{CT+1} \left[ \int_0^T G(u)+ C \,dt\right]\leq 1,
 \end{eqnarray*}
from which the result follows in the case $\|u\|_{L^{G}}=1$. The general case is obtained by normalization.
\end{proof}

Let $G:\mathbb{R}^{2n}\to [0,\infty)$ be a semi-symplectic $G$-function. Then Proposition \ref{incrustsimplec}, together with Hölder's inequality
\begin{equation}\label{holder}
 \left|\int_0^T\langle v, u\rangle dt \right|\leq 2\|u\|_{L^G}\|v\|_{L^{G^\star}}, \quad  \text{for every } u\in L^G, v\in L^{G^\star},
\end{equation}
implies that the \emph{bilinear form}\index{bilinear form}
 \[ \omega[u,v]:= \int_0^T \left<Jv,u\right>\,dt,\]
is well defined on $L^G([0,T],\mathbb{R}^{2n})\times L^G([0,T],\mathbb{R}^{2n})$. Moreover, $\omega$ is alternating, that is, $\omega[u,v]=-\omega[v,u]$, and non-degenerate, i.e., if $\omega[u,v]=0$ for all $v$, then $u=0$. Hence, according to Definition \ref{def:espaciosimplectico}, $\left(L^G([0,T],\mathbb{R}^{2n}),\omega\right)$ is an infinite-dimensional symplectic space. Naturally, we shall call an \emph{Orlicz symplectic space}\index{space!Orlicz symplectic} any space of the form $L^G([0,T],\mathbb{R}^{2n})$ where $G$ is a semi-symplectic $G$-function.

\section{Some additional lemmas}\label{sec:lemas}

The following proposition was introduced in \cite{acinas2019clarke} for convex and differentiable Frechet functions. We present its generalization to convex functions. The proof follows the same lines as in \cite[Prop. 2.2]{acinas2019clarke}; we present it for completeness.

\begin{Propos}\label{cota1}
Let $ \mathcal{H}:X\to \mathbb{R}$ be a convex function. Assume that there exist a convex function $G: X \to \mathbb{R}$ and constants $\beta, \gamma >0$ such that
\begin{equation}\label{acot1}
-\beta \leq  \mathcal{H}(u)\leq G(u)+\gamma \quad \text{for all } u \in X.
\end{equation}
Then, for any $r>1$, the following holds: if $\xi \in \partial  \mathcal{H}(u)$, then
\[
G^\star(\xi)\leq \frac{1}{r-1}G(ru)+\frac{r}{r-1}(\beta + \gamma).
\]
\end{Propos}

\begin{proof}
Taking conjugates in \eqref{acot1} and evaluating at $\xi$, we obtain
\[
G^\star(\xi)-\gamma\leq  \mathcal{H}^\star(\xi)=\left<\xi,u\right>- \mathcal{H}(u).
\]

By \ref{G:young-ine}, we have
\[
\left<\xi,u\right>=\frac{1}{r}\left<\xi,r u\right>\leq \frac{1}{r}G^\star(\xi)+\frac{1}{r}G(ru).
\]

Combining these inequalities and using \eqref{acot1}, we obtain
\[
G^\star(\xi)\leq \frac{1}{r}G^\star(\xi)+\frac{1}{r}G(ru)+\beta+\gamma,
\]
from which the desired conclusion follows.
\end{proof}
We recall the anisotropic Poincar\'e--Wirtinger inequality (see Lemma 2.4 in \cite{Mazzone2019} and Theorem 4.4 in \cite{chamra2017anisotropic}).

\begin{lem}[Anisotropic Poincar\'e--Wirtinger inequality]
\label{lem:inclusion orliczII} Let $\phi:\rr^d\to [0,+\infty)$ be an $G$-function
and let $u\in\wphi\left([0,T],\rr^d\right)$. Then

\begin{equation}\label{eq:wirtinger}
  \phi\left(\tilde{u}(t)\right)\leq\frac{1}{T} \int_0^T \phi\left(Tu'(r)\right)\,dr.\tag{A.P-W.I}
\end{equation}

\end{lem}

We are now in a position to prove Lemma \ref{lem:form_quad}.

\begin{proof} \emph{Lemma \ref{lem:form_quad}. }
Let $u \in W^{1,G}_T([0,T],\mathbb{R}^{2n})$. Since $G$ is semi-symplectic, there exist $K>0$ and $C\geq 0$ such that $G^\star(Jx)\leq G(Kx)+C$ for all $x \in \mathbb{R}^{2n}$.

Recalling the notation $\tilde{u}=u-\bar{u}$, and using that $\int_0^T J u' \, dt=0$, in combination with \ref{G:young-ine}, the semi-symplecticity of $G$, and the anisotropic Poincaré--Wirtinger inequality (Lemma \ref{lem:inclusion orliczII}), we obtain
 \begin{eqnarray*}
 \int_0^T \left<Ju',u\right>dt&=&\frac{K}{T} \int_0^T\left<\frac{T}{K}Ju',\tilde{u}\right>dt\\
 &\geq&-\frac{K}{T}\left\{\int_0^T G^\star\left(J\frac{Tu'}{K}\right)dt+\int_0^T G(\tilde{u})dt\right\}\\
 &\geq& -\frac{K}{T} \left\{2\int_0^T G(Tu')dt + C\right\}.
 \end{eqnarray*}
\end{proof}

The following lemma is a particular case of \cite[Lemma 6.2]{RUF2023109996}.

\begin{lem}\label{lem:6.2ruf} Suppose that $u$ is a measurable function and that there exists $K>0$ such that
\begin{equation}\label{ineq:dualidad}
 \left|\int\langle u,\varphi\rangle dt\right|\leq K\|\varphi\|_{L^G},
\end{equation}
for every $\varphi\in L^\infty([0,T],\rr^d)$. Then $u\in L^{G^\star}$ and $\|u\|_{L^{G^\star}}\leq K$.
\end{lem}

\section{Clarke duality}\label{sec:3}

The aim of this section is to establish Clarke’s duality principle for Hamiltonians bounded by $G$-functions and possibly not differentiable in the Fréchet sense; see \cite{mawhin2010critical} for a classical version of this result and \cite{acinas2019clarke} for the case of Fréchet-differentiable Hamiltonians and $G$-functions.

In the proof of the main theorem in this section, we will use \cite[Cor. 4.5]{demaria2025}. The latter is formulated in terms of Clarke's subdifferential (denoted by $\partial_C$), which is a more general notion of differentiability than that arising from convex analysis (see \cite{clarke2013functional}). In particular, when a function $F$ is convex and lower semicontinuous, both subdifferentials coincide on the effective domain of $F$ (see \cite[Th. 10.8]{clarke2013functional}).

When $f(x,y)$ is a real-valued function defined on a Cartesian product of Banach spaces $X_1\times X_2$, we denote by $\partial_x f$ and $\partial_y f$ the partial subdifferentials of $f$, i.e., the subdifferentials with respect to each variable while the other is held fixed. It is easy to see from the definitions that the following relation holds:
 $$\partial f\subset \partial_x f\times \partial_y f.$$

As a consequence of this relation, when $\xi\in \partial f$, we write $\xi=(\xi_x,\xi_y)$, with $\xi_x\in \partial_x f$ and $\xi_y\in \partial_y f$.

\begin{Def} We define the space $E^G$ as the closure of $L^\infty$ in the space $L^G$.
 \end{Def}

\begin{remark} It is always true that $E^G$ is a subspace of $L^G$. The equality $L^G=E^G$ holds when $G\in \Delta_2$.
We know from \cite[Th. 6.4-Cor. 6.2]{RUF2023109996} and \cite[Th. 5.4]{orliczvectorial2005} that $L^{G^\star}=\left(E^G\right)^\star$.
\end{remark}

\begin{teo}[Clarke duality]\label{teodifacciondual} Let $G:\mathbb{R}^{2n}\to [0,\infty)$ be a $G$-function, with $G^\star$ semi-symplectic, and assume that the following conditions hold:

\begin{enumerate}
    \item $ \mathcal{H}: [0,T]\times \mathbb{R}^{2n}\to \mathbb{R}$ is measurable in $t$ for each $u \in \mathbb{R}^{2n}$, and moreover $ \mathcal{H}(t,\cdot)$ is convex for almost every $t \in [0,T]$.
    \item\label{hipcotaHconG}  There exist $\beta, \gamma \in L^1 \text{ and } \Lambda>\lambda>0 $ such that
    \begin{equation}\label{acota}
        G(\lambda y)-\beta(t)\leq  \mathcal{H}(t,y)\leq G(\Lambda y)+\gamma(t),
    \end{equation}
for almost every $ t \in [0,T]$ and all $y\in \mathbb{R}^{2n}$.
\end{enumerate}
Then, the dual action \eqref{eq:accion_dual} is differentiable in the sense of convex analysis and satisfies
\begin{equation}\label{diffacciondual}
 \partial\chi(v) \subset \int_0^T \partial \left( \frac{1}{2}\left<J v', v\right> +  \mathcal{H}^\star (t,v')\right)\,dt,
\end{equation}
in the space $W^{1, G^\star}_T([0,T],\mathbb{R}^{2n})\cap \{v:v' \in \Pi(E^{G^\star}, \lambda)\}$.

Moreover, if $v$ is a critical point of $\chi$ with $v' \in \Pi(E^{G^\star}, \lambda)$, then there exists $u \in W^{1, G}_T([0,T],\mathbb{R}^{2n})$ such that, for almost every $t$, $u(t) \in \partial  \mathcal{H}^\star(t,v'(t))$, and $u$ solves

\begin{equation}\label{ecuacionproblemacapdual}
 \left\{ \begin{array}{lcc}
             u'(t) \in J\partial  \mathcal{H}(t,u(t)) \\
             \\ u(0)=u(T). \\

             \end{array}
   \right.
\end{equation}
In addition, the relation $u'=Jv'$ holds.
\end{teo}
\begin{proof}

Consider the Lagrangian function $\mathcal{L}:[0,T]\times \mathbb{R}^{2n}\times \mathbb{R}^{2n}\to \mathbb{R}$, defined by
\[\mathcal{L}(t,x,y)=\frac{1}{2}\left<J y, x\right> +  \mathcal{H}^\star (t,y).\]
It is easy to see that, for each $(x,y) \in \mathbb{R}^{2n}\times \mathbb{R}^{2n}$, the function $t\mapsto \mathcal{L}(t,x,y)$ is measurable. To prove the measurability of $\mathcal{H}^\star$ with respect to $t$, we use the separability of $\mathbb{R}^{2n}$ to reduce the supremum in \eqref{eq:conjugada} to a supremum over a countable set.

If we take conjugates in \eqref{acota} and use well-known properties of convex functions (see \cite[Subsection 2.1]{acinas2019clarke}), we obtain:
\begin{equation}\label{desigconjug}
   -\gamma(t)\leq G^\star\left(\frac{y}{\Lambda}\right)-\gamma(t)\leq  \mathcal{H}^\star(t,y)\leq G^\star\left(\frac{y}{\lambda}\right)+\beta(t),
\end{equation}
for almost every $ t \in [0,T]$ and all $y\in \mathbb{R}^{2n}$. This shows that $\mathcal{H}^\star$ is finite for almost every $t$. Additionally, since $\mathcal{H}^\star(t,\cdot)$ is convex, it follows from \cite[Corollary 2.35]{clarke2013functional} that $\mathcal{H}^\star(t,\cdot)$ is locally Lipschitz with respect to $y$ for almost every $t$. Consequently, $\mathcal{L}$ is locally Lipschitz for almost every $t$.

Finally, $\mathcal{L}$ is regular (see \cite[Def. 10.12]{clarke2013functional}) for almost every $t$ and for all $(x,y)$, since it is the sum of a Fréchet differentiable function and a convex function (see \cite[Prop. 2.3.6]{clarke1990optimization}). Using \ref{G:young-ine} and \eqref{desigconjug}, we have

 \begin{equation}\label{eq:cota_L}
  \begin{split}
|\mathcal{L}(t,x,y)| &\leq  \lambda \left(G^\star\left(\frac{y}{\lambda}\right)+G(J^tx)\right) +G^\star\left(\frac{y}{\lambda}\right)+\beta(t) \\
                     &\leq  C(x) \left[G^\star\left(\frac{y}{\lambda}\right) + b(t)\right],
  \end{split}
\end{equation}
for almost every $ t \in [0,T]$ and all $y\in \mathbb{R}^{2n}$, where $C$ is a continuous function of $x$ and $b\in L^1$.

Since $\mathcal{L}(t,x,y)$ is Fréchet differentiable with respect to the variable $x$, if $\xi \in \partial\mathcal{L}(t,x,y)$, then $\xi_{x}=\frac{1}{2}Jy$.
From the convexity of $G^\star$ and the fact that $G^\star(0)=0$, it follows that, for all $y\in\mathbb{R}^{2n}$, $|y|\leq C(G^\star(y)+1)$. Therefore
\[|\xi_{x}|= \frac{1}{2}|Jy| \leq |y|\leq \lambda C\left[G^\star\left(\frac{y}{\lambda}\right)+1\right].\]

With respect to $\xi_y$, we observe that there exists $\zeta \in \partial_y  \mathcal{H}^\star(t,y)$ such that
\begin{equation}\label{subdifLrespectoy}
  \xi_{y}=-\frac{1}{2}Jx+\zeta,
\end{equation}

and hence

\[G\left(\frac{\lambda \xi_{y}}{2}\right)\leq\frac{1}{2}G\left(\frac{\lambda Jx}{2}\right)+\frac{1}{2}G(\lambda \zeta).\]
Applying Proposition \ref{cota1} and \eqref{desigconjug}, we obtain
\begin{equation}\label{inec26}
    G(\lambda \zeta) \leq \frac{1}{r-1}G^\star\left(r\frac{y}{\lambda}\right)+\frac{r}{r-1}(\beta(t)+\gamma(t)).
\end{equation}
Thus
\begin{eqnarray*}
G\left(\frac{\lambda\xi_{y}}{2}\right)&\leq&\frac{1}{2}G\left(\frac{\lambda Jx}{2}\right)+\frac{1}{2(r-1)}G^\star\left(r\frac{y}{\lambda}\right)+\frac{r}{2(r-1)}(\beta(t)+\gamma(t))\\
&\leq&  \frac{1}{2} \max \left\{G\left(\frac{\lambda Jx}{2}\right),\frac{r}{r-1}\right\}\left[G^\star\left(r\frac{y}{\lambda}\right)+(\beta(t)+\gamma(t)+1)\right].
\end{eqnarray*}
In summary, we have proved that there exist $\lambda_0, \Lambda_0 > 0$, $0\leq b \in L^1$ and a continuous function $a:\mathbb{R}^{2n} \rightarrow [0,+\infty)$ such that, for a.e. $t \in [0,T]$, if
 $\xi=(\xi_{x},\xi_{y}) \in \partial_{x}\mathcal{L}(t,x,y)\times\partial_{y}\mathcal{L}(t,x,y)$, then
      \begin{equation}\label{Hip4}
          |\mathcal{L}(t,x,y)| + |\xi_{x}|+G \left(\frac{\xi_{y}}{\lambda_0}\right) \leq a(x) \left(G^\star \left(\frac{y}{\Lambda_0}\right)+ b(t)\right)
     \end{equation}

Therefore, the function $\mathcal{L}$ satisfies all the hypotheses of \cite[Cor. 4.5]{demaria2025}. Hence, \eqref{diffacciondual} holds for all $v \in W^{1,G^\star}_T\cap \{v:v' \in \Pi(E^{G^\star}, \Lambda_0)\}$. Since $\Lambda_0=\lambda/r$ and $r$ can be chosen arbitrarily close to one, \eqref{diffacciondual} holds for all $v \in W^{1,G^\star}_T\cap \{v:v' \in \Pi(E^{G^\star}, \lambda)\}$. This concludes the proof of the differentiability of the dual action.

Let $v \in W^{1,G^\star}_T \cap \{v:v' \in \Pi(E^{G^\star}, \lambda)\}$ be a critical point of $\chi$. Then $0 \in \partial \chi(v)$, that is, for almost every $t$ there exists $ \xi(t)=(\xi_{x}(t),\xi_{y}(t)) \in \partial \mathcal{L}(t,v(t),v'(t))$ such that for all $h \in  W^{1,G^\star}_T([0,T], \mathbb{R}^{2n})$ it holds that

\[0=\int_0^T \left<\xi_{x}(t),h(t)\right>+\left<\xi_{y}(t),h'(t)\right>\,dt.\]

Taking into account that $\xi_{x}(t)=\frac{1}{2}Jv'(t)$ and that, by \eqref{subdifLrespectoy}, $\xi_{y}(t)=-\frac{1}{2}Jv(t)+u(t)$ with $u(t) \in \partial  \mathcal{H}^\star(t,v'(t))$, we obtain
\[0=\int_0^T \frac{1}{2}\left<Jv'(t),h(t)\right>-\frac{1}{2}\left<Jv(t),h'(t)\right>+\left<u(t),h'(t)\right>\,dt,\]
from which we conclude
\[\int_0^T\left<u(t)-\frac{1}{2}Jv(t),h'(t)\right> \, dt= -\int_0^T \frac{1}{2}\left<Jv'(t),h(t)\right> \, dt.\]

According to Proposition \ref{incrustsimplec}, both $-\frac{1}{2}Jv$ and $Jv'$ belong to $L^G$. Moreover, from inequality \eqref{inec26}, taking $r>1$ such that $d(v', E^{G^\star})<\lambda/r$, it follows from \cite[Eq. (6)]{Mazzone2019} that $u \in L^G$. In particular, $u-\frac{1}{2}Jv$ and $Jv'$ belong to $L^1$. This shows that $\frac{1}{2} Jv'$ is the weak derivative of $u-\frac{1}{2}Jv$, and therefore we can write $u'-\frac{1}{2}Jv'=\frac{1}{2}Jv'$, or equivalently $u'=Jv'$.

Moreover, from \ref{G:young-eq}, we know that the relation $u(t) \in \partial  \mathcal{H}^\star(t,v'(t))$ implies that $v'(t) \in \partial  \mathcal{H}(t,u(t))$. This shows that the differential inclusion in \eqref{ecuacionproblemacapdual} holds.

On the other hand,
\[
u(T)-u(0)=\int_0^T u'(t)\,dt=\int_0^T Jv'(t)\,dt=J\int_0^T v'(t)\,dt=0,
\]
which shows that the boundary condition in \eqref{ecuacionproblemacapdual} is satisfied.
\end{proof}

\section{Proof main theorem}

We are now in a position to prove our main Theorem \ref{teominimoacciondual}. We split the proof into six steps.

\paragraph{Step 1:} \emph{Estimating the perturbed dual action.}

Let $0<r<1$ be sufficiently small such that
\[\Lambda^{-1}>(1+r)T \max\{1,C_1/2\},\]
and take $\epsilon>0$ with $\epsilon<r\Lambda$.
We define the \emph{perturbed Hamiltonian} as
\[ \mathcal{H}_{\epsilon}(t,u):=  \mathcal{H}(t,u)+G(\epsilon u).\]

By item 1 of Theorem \ref{teominimoacciondual}, inequality \ref{G:young-ine}, and \ref{G:supra_adit}, we have
\begin{equation}\label{des32}
  \begin{split}
    \mathcal{H}_{\epsilon}(t,u)&\geq \left<\xi(t),u\right>+G(\epsilon u)\\
&\geq - G^\star\left(\frac{1}{r \epsilon}\xi(t)\right)-G(r\epsilon u)+G(\epsilon u)\geq G((1-r)\epsilon u)-\beta(t),
\end{split}
\end{equation}
where $\beta(t):=G^\star\left(\frac{1}{r \epsilon}\xi(t)\right)$. Since $G^\star \in \Delta_2$ and $\xi \in L^{G^\star}$, it follows from \cite[Theorem 3.2]{orliczvectorial2005} that $\beta \in L^1([0,T],\mathbb{R})$.

On the other hand,
\begin{equation}\label{des33}
   \mathcal{H}_{\epsilon}(t,u)\leq G(\Lambda u) + \alpha(t) +G(\epsilon u)\leq G((1+r)\Lambda u)+\alpha(t).
\end{equation}

From \eqref{des32}, \eqref{des33}, and the properties of the Fenchel conjugate, we obtain
\begin{equation}\label{inec34}
 G^\star\left(\frac{v}{(1+r)\Lambda}\right)-\alpha(t)\leq  \mathcal{H}^\star_{\epsilon}(t,v)\leq G^\star\left(\frac{v}{(1-r)\epsilon}\right)+\beta(t).
\end{equation}

We consider the dual action $\chi_{\epsilon}:W^{1,G^\star}_T([0,T],\mathbb{R}^{2n})\to \mathbb{R}$ associated with $\mathcal{H}_{\epsilon}$, i.e.
\[\chi_{\epsilon}(v):=\int_0^T\frac{1}{2}\left<Jv',v\right> +  \mathcal{H}^\star_{\epsilon}(t,v')\,dt.\]

From \eqref{inec34}, \eqref{inec16}, and using that $T(1+r)\Lambda<1$, we have

\begin{multline}\label{eq:estima_dual}
\chi_{\epsilon}(v)\geq-\frac{C_1}{2}\int_0^TG^\star(Tv')dt+\int_0^T G^\star\left(\frac{v'}{(1+r)\Lambda}\right)dt-\int_0^T \alpha(t)dt-C_2\\
\geq -\frac{C_1}{2}\int_0^TG^\star(Tv')dt+\frac{1}{T(1+r)\Lambda}\int_0^T G^\star(Tv')dt-\int_0^T \alpha(t)dt-C_2\\
:= C_{\chi}\int_0^T G^\star(Tv')dt-B_{\chi}.
\end{multline}
By the hypothesis on $\Lambda$ and our choice of $r$, we infer that $C_{\chi}>0$.

\paragraph{Step 2:} \emph{Solving perturbed problem via Clarke duality.}

Since $\chi_{\epsilon}(v)=\chi_{\epsilon}(v+c)$ for every $c\in \mathbb{R}^{2n}$, it is sufficient to minimize $\chi_{\epsilon}$ on the set
\[\widetilde{W}^{1,G^\star}_T:=\left\{v \in W^{1,G^\star}_T:\int_0^T v(t)\, dt=0\right\}.\]

The dual action is coercive on this space. To see this, let $\{v_n\}\subset \widetilde{W}^{1,G^\star}_T$ and suppose that $\|v_n\|_{W^{1,G^\star}}\to \infty$. Then, by Proposition \ref{propos:noprmas_equi} and since $\bar{v}_n=0$, it follows that $\|v'_n\|_{L^{G^\star}}\to \infty$. Thus, from inequality \eqref{eq:desi_norm_mod}, $\int_0^TG^\star(Tv'_n)\,dt\to \infty$ and, consequently, by \eqref{eq:estima_dual}, $\chi_{\epsilon}(v_n)\to \infty$.

It follows from the above argument that if $\{v_n\}\subset\widetilde{W}^{1,G^\star}_T([0,T],\mathbb{R}^{2n})$ is a minimizing sequence for $\chi_{\epsilon}$, then $\|v_n\|_{W^{1,G^\star}}$ is bounded. As a consequence, there exists a subsequence (which we still denote by $v_n$) that converges uniformly to a function $v_\epsilon$ (see Proposition 3.4 in \cite{demaria2025}). We recall that $L^{G^\star}=\left(E^G\right)^\star$. From \cite[Th. 6.3]{orliczvectorial2005}, we know that $E^{G}$ is separable. Since $\{v'_n\}$ is a bounded sequence in $L^{G^\star}$, by \cite[Cor. 3.30]{brezis2011functional} there exists a subsequence $\{v_n\}$ and $w$ in $L^{G^\star}$ such that $\{v'_n\}$ converges weak$^\star$ to $w$. It is easy to see that $w$ is an absolutely continuous function and that $w=v_\epsilon'$ a.e. Consequently, $v_\epsilon\in W^{1,G^\star}_T([0,T],\mathbb{R}^{2n})$, and by the uniform convergence of $v_n$ to $v_\epsilon$ we have that $v_\epsilon(0)=v_\epsilon(T)$ and $\int_0^T v_\epsilon(t)\,dt=0$. In conclusion, $v_\epsilon\in\widetilde{W}^{1,G^\star}_T([0,T],\mathbb{R}^{2n})$.

On the other hand, using \cite[Th. 3.6]{buttazzo1998one}, it follows that $\mathcal{I}:L^1\times L^1 \to \mathbb{R}$, defined by
\[\mathcal{I}(u,v):=\int_0^T\frac{1}{2}\left<Jv(t),u(t)\right> +  \mathcal{H}^\star_{\epsilon}(t,u(t))\,dt,\]
is sequentially lower semicontinuous when $L^1\times L^1$ is equipped with the product topology $\mathscr{T}$ corresponding to the strong topology on the first factor and the weak topology on the second.

In our case, we know that $v_n$ converges uniformly to $v_\epsilon$ and that $v'_n$ converges weak$^\star$ to $v_\epsilon'$ in $L^{G^\star}$. Hence, $(v_n,v'_n)$ converges to $(v_\epsilon,v_\epsilon')$ in $\mathscr{T}$.

With this, applying \cite[Th. 3.6]{buttazzo1998one}, we obtain
\[\chi_\epsilon(v_\epsilon)=\mathcal{I}(v_\epsilon,v_\epsilon')\leq \lim_{n\to \infty}\chi_\epsilon(v_n)=\inf_{v\in\widetilde{W}^{1,G^\star}_T}\chi_\epsilon(v).\]
Thus, $v_{\epsilon} \in \widetilde{W}^{1,G^\star}_T([0,T],\mathbb{R}^{2n})$ is a minimizer of $\chi_{\epsilon}$.

Since $G^\star \in \Delta_2$, we know from \cite[Th. 5.2]{orliczvectorial2005} that $L^\infty$ is dense in $L^{G^\star}$. Then, by Theorem \ref{teodifacciondual}, there exists $u_\epsilon \in W^{1,G}_T([0,T],\mathbb{R}^{2n})$ such that for almost every $t$, $u_{\epsilon}(t)\in \partial  \mathcal{H}^\star_{\epsilon}(t,v'_{\epsilon}(t))$, and $u_\epsilon$ solves

\begin{equation}\label{ec37}
\left\{ \begin{array}{lcc}
             u'_{\epsilon} (t)\in J\partial  \mathcal{H}_{\epsilon}(t,u_{\epsilon}(t)) \\
             \\ u_{\epsilon}(0)=u_{\epsilon}(T) \\
             \end{array}
   \right.
\end{equation}
and, furthermore, we have the relation $u'_{\epsilon}=Jv'_{\epsilon}$, or equivalently
\begin{equation}\label{ec39}
J v_{\epsilon}=u_{\epsilon}-\bar{u}_{\epsilon}.
\end{equation}
\noindent \textbf{Step 3:} \emph{Estimating $v_\epsilon$.}

The function $\bar{\mathcal{H}}:\mathbb{R}^{2n}\to \mathbb{R}$, defined by
\[
x\mapsto \int_0^T  \mathcal{H}(t,x)\,dt,
\]
is convex and coercive on a finite-dimensional space; therefore, $\bar{\mathcal{H}}$ attains a minimum at some $x_0\in \mathbb{R}^{2n}$. Consequently, $0 \in \partial \bar{\mathcal{H}}(x_0)$.

Since the hypotheses of \cite[Th. 2.7.1]{clarke1990optimization} are satisfied, there exists a function $\eta(t) \in \partial  \mathcal{H}(t,x_0)$ such that, for all $x \in \mathbb{R}^{2n}$, the mapping $t\mapsto \langle \eta(t),x\rangle$ belongs to $L^1$ and
\[
0=\int_0^T \langle \eta(t),x\rangle\,dt
=\left\langle\int_0^T \eta(t)\, dt,x\right\rangle
\quad \forall \, x \in \mathbb{R}^{2n},
\]
from which it follows that
\[
\int_0^T \eta(t)\, dt=0.
\]

We define
\[
w(s)=\int_0^s\eta(t)\, dt+C,
\]
where $C$ is a constant chosen so that $\int_0^T w(s)\,ds=0$. The function $w$ is absolutely continuous. We now show that $w \in \widetilde{W}^{1,G}_T([0,T],\mathbb{R}^{2n})$.

By items 1 and 2 of Theorem \ref{teominimoacciondual}, it follows that for any $t \in [0,T]$ and $y \in \mathbb{R}^{2n}$
\[
-G^\star\left(\frac{\xi(t)}{\Lambda}\right)
\leq  \mathcal{H}(t,y)+G(\Lambda y)
\leq 2G(\Lambda y) + \alpha(t).
\]

The functions $\mathcal{H}(t,y)+G(\Lambda y)$ and $2G(\Lambda y)$ satisfy the hypotheses of Proposition \ref{cota1}. Taking $r=2$, we have that if $\zeta \in \partial  \mathcal{H}(t,y)$, then
\[
G^\star \left(\frac{\zeta + \Lambda \nabla G(\Lambda y)}{2\Lambda}\right)
\leq G(2\Lambda y) +  G^\star\left(\frac{\xi(t)}{\Lambda}\right) + \alpha(t).
\]

From this inequality, taking into account that $w'=\eta$, Fenchel's identity, and \ref{G:d(2x)}, we obtain
\begin{eqnarray*}
G^\star \left(\frac{w'}{4\Lambda}\right)
&=& G^\star \left(\frac{\eta}{4\Lambda}\right) \\
&\leq& \frac{1}{2} G^\star \left(\frac{\eta+\Lambda\nabla G(\Lambda x_0)}{2\Lambda}\right)
+ \frac{1}{2}G^\star\left(\frac{\nabla G(\Lambda x_0)}{2}\right)\\
&\leq& \frac{1}{2}\left[G(2\Lambda x_0)+G^\star\left(\frac{\xi(t)}{\Lambda}\right)+\alpha(t)\right]
+\frac{1}{4} \Lambda x_0\cdot \nabla G(\Lambda x_0)\\
&\leq& G(2\Lambda x_0)
+\frac{1}{2}G^\star\left(\frac{\xi(t)}{\Lambda}\right)
+\frac{1}{2}\alpha(t)
\in L^1.
\end{eqnarray*}
Thus $w' \in L^{G^\star}$.

Now, using Fenchel's identity
\[
\mathcal{H}^\star(t,w')=\langle w',x_0\rangle- \mathcal{H}(t,x_0),
\]
we deduce that $\mathcal{H}^\star(\cdot,w'(\cdot))\in L^1([0,T],\mathbb{R})$.

From the inequality $\mathcal{H}(t,u)\leq  \mathcal{H}_{\epsilon}(t,u)$, we deduce that $\mathcal{H}^\star_{\epsilon}(t,v)\leq  \mathcal{H}^\star(t,v)$. Using this and the bound from Step~1 for $\chi_{\epsilon}$, we obtain
\begin{multline*}
C_{\chi}\int_0^T G^\star(Tv'_{\epsilon})\,dt - B_{\chi}
\leq \chi_{\epsilon}(v_{\epsilon})
\leq \chi_{\epsilon}(w) \\
\leq \int_0^T \frac{1}{2}\langle Jw',w\rangle
+ \mathcal{H}^\star(t,w')\,dt
=:c_1<\infty.
\end{multline*}
\paragraph{Step 4:} \emph{Estimating $u_\epsilon$.}

Since $G^\star$ is semi-symplectic, by Proposition \ref{incrustsimplec} and Step 3 we obtain that $u'_{\epsilon}$ is bounded in $L^G$. Then, by the anisotropic Poincaré--Wirtinger inequality, it follows that $\tilde{u}_{\epsilon}$ is uniformly bounded in $L^\infty$. Furthermore, there exists $C_3$ such that
\begin{equation}\label{cotaGconutilde}
 \int_0^T G(\Lambda \tilde{u}_{\epsilon})\,dt\leq C_3.
\end{equation}
Thus, using inequality \eqref{inec16}, we have
\begin{multline}\label{ec40}
\int_0^T\langle Ju'_{\epsilon},u_{\epsilon}\rangle\,dt
= \int_0^T \langle -v'_{\epsilon},J v_{\epsilon}+\bar{u}_{\epsilon}\rangle\,dt\\
\geq -C_4\int_0^T G^\star(Tv'_{\epsilon})\,dt - C_5 \geq -C_6.
\end{multline}

By \eqref{ec37}, the convexity of $\mathcal{H}(t,\cdot)$, and well-known properties of subdifferentials (see \cite[Th. 10.8, Th. 10.13]{clarke2013functional}), we obtain
\begin{equation}\label{eq:forma_der}
 \exists\, \xi_{\epsilon} \in \partial  \mathcal{H}(t,u_{\epsilon}) \text{ such that }
 u'_{\epsilon}=J\big(\xi_{\epsilon}+\epsilon \nabla G(\epsilon u_{\epsilon})\big).
\end{equation}
Therefore, taking into account the second assumption in Theorem \ref{teominimoacciondual} and the fact that $\langle x,\nabla G(x)\rangle\geq 0$ for any $x \in \mathbb{R}^{2n}$, we deduce
\begin{eqnarray*}
2 \mathcal{H}\left(t, \frac{\bar{u}_{\epsilon}}{2}\right)
&\leq&  \mathcal{H}(t,u_{\epsilon})+  \mathcal{H}(t,-\tilde{u}_{\epsilon})\\
&\leq& \langle \xi_{\epsilon},u_{\epsilon}\rangle+  \mathcal{H}(t,0)
+G(\Lambda \tilde{u}_{\epsilon})+\alpha(t)\\
&\leq& \langle -J u'_{\epsilon}-\epsilon \nabla G(\epsilon u_{\epsilon}),u_{\epsilon}\rangle
+ \mathcal{H}(t,0)+G(\Lambda \tilde{u}_{\epsilon})+\alpha(t)\\
&\leq& \langle -J u'_{\epsilon},u_{\epsilon}\rangle
+ \mathcal{H}(t,0)+G(\Lambda \tilde{u}_{\epsilon})+\alpha(t).
\end{eqnarray*}
Integrating the previous inequality and using \eqref{ec40} and \eqref{cotaGconutilde}, it follows that
\[
\int_0^T  \mathcal{H}\left(t,\frac{\bar{u}_{\epsilon}}{2}\right)\,dt\leq C_7.
\]

Now, by the third assumption in the hypotheses of our theorem, it follows that $\bar{u}_{\epsilon}$ is uniformly bounded. Consequently, $u_{\epsilon}$ is uniformly bounded in $W^{1,G}([0,T],\mathbb{R}^{2n})$.

\textbf{Step 5:} \emph{Solving  problem \eqref{problemadualfinal}.}

As in Step 2, we obtain a sequence $u_{\epsilon_n}$ and a function $u\in W^{1,G}([0,T],\mathbb{R}^{2n})$ such that $u_{\epsilon_n}$ converges uniformly to $u$ and $u'_{\epsilon_n}$ converges in the weak$^\star$ topology of $L^G$ to $u'$. We now show that $u$ solves our problem.

First, we prove that $\xi_{\epsilon_n} \overset{\star}{\rightharpoonup} -Ju'$ in $L^{G^\star}$. Taking into account \eqref{eq:forma_der}, it is sufficient to show that $Ju'_{\epsilon_n} \overset{\star}{\rightharpoonup} Ju'$ and $\epsilon_n \nabla G (\epsilon_n u_{\epsilon_n}) \overset{\star}{\rightharpoonup} 0$ in $L^{G^\star}$.

Note that $J$ does not necessarily induce an embedding from $L^G$ into $L^{G^\star}$. However, it is true that $J: L^G \hookrightarrow L^1$. Indeed, if $w\in L^G$ and $\varphi \in L^\infty$, then, since $J$ induces an isometry on $L^\infty$, we have
\[
\|Jw\|_{L^1} = \sup_{\|\varphi\|_{L^\infty}\leq 1}\int_0^T\langle Jw,\varphi\rangle dt
= \sup_{\|\varphi\|_{L^\infty}\leq 1}\int_0^T\langle w,J\varphi\rangle dt
\leq K \|w\|_{L^G}.
\]

Therefore, using that $Ju'\in L^1$ and that $u'_{\epsilon_n}\overset{\star}{\rightharpoonup} u'$ in $L^G$, we obtain, for every $\varphi \in L^\infty$,

\begin{multline}\label{convergencia2} -\int_0^T\left<Ju'_{\epsilon_n},\varphi\right>dt=\int_0^T\left<u'_{\epsilon_n},J\varphi\right>dt\to \int_0^T\left<u',J\varphi\right>dt =-\int_0^T\left<Ju',\varphi\right>dt \end{multline}

At this point, note that we do not know if $Ju'\in L^{G^\star}$. This follows from Lemma \ref{lem:6.2ruf} and from the fact that the following inequality holds for every $\varphi\in L^\infty$:
\[
\left|\int_0^T\langle Ju',\varphi\rangle dt\right|
\leq \sup_n \left|\int_0^T\langle Ju'_{\epsilon_n},\varphi\rangle dt\right|
\leq K \|\varphi\|_{L^G},
\]
where $K$ is a constant depending on the uniform bounds of the functions $v_{\epsilon}$ obtained in Step 3.

The term $\epsilon_n \nabla G (\epsilon_n u_{\epsilon_n})$ converges uniformly to $0$. To see this, it is enough to observe that $u_{\epsilon_n}$ converges uniformly to $u$ and that $\nabla G$ is locally bounded, since $G$ is locally Lipschitz.

Now we show that $-Ju'\in \partial  \mathcal{H}(t,u)$ for a.e. $t\in[0,T]$. Let $v\in L^{G}$. Since $\xi_n(t) \in \partial  \mathcal{H}(t,u_{\epsilon_n}(t))$, by definition we have:
\begin{equation}\label{subdiferencial1}
    \left<\xi_n(t),v(t)-u_{\epsilon_n}(t)\right>+ \mathcal{H}(t,u_{\epsilon_n}(t))\leq  \mathcal{H}(t,v(t)).
\end{equation}

Integrating this inequality and taking into account the uniform boundedness of the functions $u_{\epsilon_n}$, the bounds \eqref{eq:cota_sup_H} and \eqref{eq:cota_inf_H}, the fact that $\xi_n\overset{\star}{\rightharpoonup} -Ju'$, and the Lebesgue Dominated Convergence Theorem, we conclude that
\[
\int_0^T\left<-Ju',v-u\right>dt +\int_0^T  \mathcal{H}(t,u)dt\leq \int_0^T  \mathcal{H}(t,v)dt.
\]
In other words, $-Ju'\in \partial \mathcal{I}(u)$, where
\[
\mathcal{I}(v)= \int_0^T  \mathcal{H}(t,v)dt.
\]

We now prove that $\mathcal{H}$ satisfies the hypotheses of \cite[Cor. 4.5]{demaria2025}. From this result, and using standard arguments, we conclude that
\[
u'(t)\in J\partial  \mathcal{H}(t,u(t)) \quad \text{for almost every } t\in [0,T].
\]
Furthermore, from the uniform convergence of $u_{\epsilon_n}$ we conclude that $u$ satisfies the boundary condition $u(0)=u(T)$, that is, $u$ is a solution of \eqref{problemadualfinal}.

Consequently, let us verify that the hypotheses of \cite[Cor. 4.5]{demaria2025} are satisfied. The only one that does not follow immediately from the assumptions on $\mathcal{H}$ is the following:
\begin{equation}\label{Hip4reformulada}
        \forall x \in \mathbb{R}^{2n},\ \forall \bar{\xi} \in \partial  \mathcal{H}(t,x): \quad |\bar{\xi}| \leq C(x)\, b(t),
\end{equation}
where $b(t)$ is an integrable function and $C(x)$ is bounded on bounded sets.

Using \eqref{eq:cota_inf_H} and \ref{G:young-ine}, we obtain
\[
\mathcal{H}(t,x)\geq \left<\xi(t),x\right>\geq -G^\star\left(\xi(t)\right)- G(x)=-\beta(t)-G(x).
\]
It follows that $\beta(t)= G^\star\left(\xi(t)\right) \in L^1$, since $\xi\in L^{G^{\star}}$ and $G^\star\in\Delta_2$.

From the hypotheses, we also have that
\begin{equation}\label{cotasub2}
 \mathcal{H}(t,x) +  G(x)\leq 2 G(b x)+\alpha(t),
\end{equation}
where $b\geq \max\{\Lambda,1\}$. Let $\bar{\xi} \in \partial \mathcal{H}(t,x)$. Then we can apply Proposition \ref{cota1} to the function $\mathcal{H}+G$ instead of $\mathcal{H}$, taking $r=2$, to obtain
\[
G^\star\left(\frac{\bar{\xi}+  \nabla G(x)}{2b}\right)\leq  G\left(2bx\right)+ \beta(t) + \alpha(t).
\]

Using inequality \ref{G:|x|<=G}, there exists $C>0$ such that
\[
C\left|{\bar{\xi}+  \nabla G(x)}\right| \leq  G\left(2bx\right)+  \beta(t) + \alpha(t)+1.
\]
This inequality implies that $|\bar{\xi}|\leq a_1(x)b_1(t)$, where
\[
b_1(t)=C^{-1}(\beta(t) +\alpha(t)+1)+1,\quad\text{and}\quad a_1(x)= C^{-1}G(2bx) +|\nabla G(x)|+1.
\]

Since $|\nabla G(x)|$ is bounded on bounded sets (because $G$ is locally Lipschitz), we conclude that the conditions of \cite[Cor. 4.5]{demaria2025} are satisfied.

\textbf{Final Step:} \emph{$v=-J\tilde{u}$ minimizes the dual action.}

Since $v'_{\epsilon_n} \in \partial  \mathcal{H}_{\epsilon_n}(t,u_{\epsilon_n})$, by Fenchel's equality and taking into account the definition of $\mathcal{H}_{\epsilon_n}$, we have
\begin{eqnarray*}
\chi_{\epsilon_n}(v_{\epsilon_n})&=&
\int_0^T \left[\frac{1}{2}\left<Jv'_{\epsilon_n},v_{\epsilon_n}\right>+\left<u_{\epsilon_n},v'_{\epsilon_n}\right>- \mathcal{H}(t,u_{\epsilon_n})-G(\epsilon_n u_{\epsilon_n})\right]dt.
\end{eqnarray*}

Since $u_{\epsilon_n}$ converges uniformly to $u$, it follows that
\[
\int_0^T \mathcal{H}(t,u_{\epsilon_n}) dt \to \int_0^T \mathcal{H}(t,u) dt
\quad\text{and}\quad
\int_0^T G(\epsilon_n u_{\epsilon_n}) dt \to 0.
\]

On the other hand,
\begin{eqnarray*}
 \left|\int_0^T\left<Jv'_{\epsilon_n},v_{\epsilon_n}\right>-\left<Jv',v\right>dt\right|
&\leq& \left|\int_0^T \left<v'_{\epsilon_n},Jv_{\epsilon_n}-Jv\right>dt\right|
 +\left|\int_0^T \left<v'_{\epsilon_n}-v',Jv\right>dt\right|\\
&\leq& \|v'_{\epsilon_n}\|_{L^{G^\star}} \|\tilde{u}_{\epsilon_n}-\tilde{u}\|_{L^{G}}
+\left|\int_0^T \left<v'_{\epsilon_n}-v',Jv\right>dt\right|.
\end{eqnarray*}

Since $v'_{\epsilon_n}$ converges weak$^\star$ to $v'$ in $L^{G^\star}$ (recall that $Ju'_{\epsilon_n} \overset{\star}{\rightharpoonup} Ju'$), it follows that $\|v'_{\epsilon_n}\|_{L^{G^\star}}$ is uniformly bounded. Consequently,
\[
\int_0^T \left<Jv'_{\epsilon_n},v_{\epsilon_n}\right>dt \to \int_0^T \left<Jv',v\right> dt.
\]

A similar argument (in fact, $\int_0^T \left<u_{\epsilon_n},v'_{\epsilon_n}\right>dt=-\int_0^T\left<Jv'_{\epsilon_n},v_{\epsilon_n}\right>dt$) shows that
\begin{eqnarray*}
 \int_0^T \left<u_{\epsilon_n},v'_{\epsilon_n}\right>dt\to \int_0^T \left<u,v'\right>dt.
\end{eqnarray*}

Using that $v'=-Ju' \in \partial  \mathcal{H}(t,u)$ and Fenchel's equality, we obtain
\begin{eqnarray}
 \lim_{n\to\infty}\chi_{\epsilon_n}(v_{\epsilon_n})
&=&\int_0^T \frac{1}{2}\left<Jv',v\right>+\left<u,v'\right>- \mathcal{H}(t,u)dt\\
 &=& \int_0^T \frac{1}{2}\left<Jv',v\right>+ \mathcal{H}^\star(t,v')dt=\chi(v).
\end{eqnarray}

Furthermore, since $\mathcal{H}^\star_{\epsilon_n}\leq  \mathcal{H}^\star$ for any $w \in W^{1,G^\star}_T([0,T],\mathbb{R}^{2n})$, it holds that
\[
\chi_{\epsilon_n}(v_{\epsilon_n})\leq\chi_{\epsilon_n}(w)\leq\chi(w).
\]
Thus, $v$ is a minimizer of $\chi$. \qed

\section{Acknowledgements}
The research has been supported by SECyT-UNRC, grant C561, FCEyN-UNLPam, grant PI 91. M. and ANPCyT PICT 2019- 3837.

\bibliographystyle{plain}
\bibliography{inclusion}

\end{document}